\documentclass[reqno]{amsart}

\usepackage{microtype}
\usepackage{amsmath,amssymb}
\usepackage[hidelinks]{hyperref}

\allowdisplaybreaks

\newtheorem{theorem}{Theorem}
\newtheorem{proposition}[theorem]{Proposition}
\newtheorem{lemma}[theorem]{Lemma}

\DeclareMathOperator{\conv}{conv}
\DeclareMathOperator{\ext}{ext}
\DeclareMathOperator{\Vol}{Vol}

\newcommand{\R}{\mathbb{R}}
\newcommand{\Z}{\mathbb{Z}}
\newcommand{\arxiv}[1]{\href{https://arxiv.org/abs/#1}{arXiv:#1}}
\newcommand{\doi}[1]{\href{https://doi.org/#1}{doi:#1}}

\title[Sharp bounds for the layer number of grids]
  {Sharp bounds for the layer number of integer grids}
\author{Shiyu Yan}
\address{University of Regensburg, 93040 Regensburg, Germany}
\email{math.shiyu.yan@gmail.com}
\date{}


\hypersetup{
  pdftitle={Sharp bounds for the layer number of integer grids},
  pdfauthor={Shiyu Yan},
  pdfsubject={Convex layers of integer grids},
  pdfkeywords={convex layers, convex-hull peeling, layer number, integer grid,
    lattice polytope, Minkowski sum}
}

\begin{document}

\begin{abstract}
The layer number of a finite point set is the number of iterations needed to
delete it by repeatedly removing the vertices of its convex hull.  Ambrus,
Hsu, Peng, and Yan conjectured that the layer number of the
$d$-dimensional integer grid $\{1,\ldots,n\}^d$ is of order
$n^{2d/(d+1)}$ for every fixed $d$.  We prove this conjecture.  Let $P_i$ be
the convex hull of the point set remaining after $i$ steps, and let $Z_n$ be the
convex hull of the lattice points in the Euclidean ball of radius $n$.  For
every step that leaves a nonempty point set, the Minkowski sum
$P_{i+1}+Z_n$ contains no vertex of $P_i+Z_n$.  Integrality of normalized
lattice volume, together with the B\'ar\'any--Larman vertex estimate for $Z_n$,
gives a lower bound, independent of $i$, on the resulting volume decrease.
Summing over $i$ yields the matching upper bound, even when $P_i$ is
lower-dimensional.  For $d\ge2$, the same upper bound holds uniformly over
all nonempty subsets of $\{1,\ldots,n\}^d$.
\end{abstract}

\maketitle
\enlargethispage{2pt}

\section{Introduction}

Let $X\subset\R^d$ be finite.  In the convex-hull peeling process, all vertices
of the convex hull of the current point set are removed at each step.  The
number of steps required to delete $X$ is its \emph{layer number}, denoted by
$\tau(X)$.  Because each deletion changes the convex hull, an estimate for a
single hull does not directly control the total number of steps.

Following \cite{AmbrusHsuPengYan}, write
\[
  [n]^d=\{1,\ldots,n\}^d
\]
for the $d$-dimensional integer grid.  We reserve $[n]^d$ for this discrete
set; its convex hull is the continuous cube $[1,n]^d$.  For $n\ge 2$, the
initial hull is therefore a cube, whereas the hulls arising later have no
prescribed shape.  An upper bound must hold uniformly throughout the process.

Dalal proved that the expected layer number of $N$ independent uniform points
in a $d$-dimensional ball is $\Theta(N^{2/(d+1)})$ \cite{Dalal}.  Setting
$N=n^d$ suggests the exponent $2d/(d+1)$ for a grid with $n^d$ points.  In
dimension two this agrees with the theorem of Har-Peled and Lidick\'y
\cite{HarPeledLidicky} that $\tau([n]^2)=\Theta(n^{4/3})$.  This agreement
between the random and lattice settings led Ambrus, Hsu, Peng, and Yan to the
following conjecture \cite[Conjecture~1]{AmbrusHsuPengYan}:
\[
  \tau([n]^d)=\Theta\!\left(n^{2d/(d+1)}\right).
\]
The same paper established the lower bound in every fixed dimension
\cite[Theorem~1]{AmbrusHsuPengYan}.  For $d\ge 3$, it also proved the upper
bound $O(n^{d-9/11})$ \cite[Theorem~4]{AmbrusHsuPengYan}.  The layer number is
invariant under translations and positive scalings.  After centering $[n]^d$
and scaling it by $(n\sqrt d)^{-1}$, we obtain a subset of the closed Euclidean
unit ball with $n^d$ points and minimum separation
$(n\sqrt d)^{-1}=\Theta((n^d)^{-1/d})$.  The rescaled grids therefore satisfy
the hypotheses of Theorem~1.6 of Ambrus, Nielsen, and Wilson
\cite{AmbrusNielsenWilson}, which gives
$\tau([n]^d)=O_d(n^2)$.  Dillon and Varadarajan subsequently proved the
explicit estimate
\[
  \tau([n]^d)\le \frac d4 n^2+1
\]
\cite{DillonVaradarajan}.  Thus, for $d\ge 3$, the best general upper bound was
quadratic in $n$.  Rote, R\"uber, and Saghafian noted in 2024 that the
asymptotic number of layers of three- and higher-dimensional boxes was still
unknown \cite[Section~1.1]{RoteRuberSaghafian}.  The conjectured exponent
$2d/(d+1)$ is smaller than $2$.  We prove the matching upper bound
$O_d(n^{2d/(d+1)})$; together with the known lower bound, this proves
Conjecture~1 of \cite{AmbrusHsuPengYan}.

\begin{theorem}\label{thm:main}
For every integer $d\ge 1$ there are constants $c_d,C_d>0$ and an integer
$N_d\ge 1$, depending only on $d$, such that every integer $n\ge N_d$
satisfies
\[
  c_d n^{2d/(d+1)}
  \le \tau([n]^d)
  \le C_d n^{2d/(d+1)}.
\]
Equivalently, for every fixed $d\ge 1$,
\[
  \tau([n]^d)=\Theta\!\left(n^{2d/(d+1)}\right).
\]
\end{theorem}

For $d\ge 3$, the new part of the theorem is the upper bound.  In fact, its
proof applies to every nonempty set $X\subset[n]^d$; the full-grid structure
is used only for the lower bound.  Let $P_i$ be the convex hull of the points
remaining from $X$ after $i$ peeling steps, and let $\Vol$ denote Euclidean
$d$-volume.  A direct attempt to track $\Vol(P_i)$ does not give a decrease of
the required order: the volume vanishes when $P_i$ is lower-dimensional, and
$P_i$ may have too few vertices even when it is full-dimensional.  We instead
fix the lattice polytope
\[
  Z_n=\conv(\Z^d\cap nB^d),
\]
where $\conv$ denotes convex hull and $B^d$ is the closed Euclidean unit ball.
Write $\beta_d=d(d-1)/(d+1)$.  The scale $n$ is natural: using $Z_r$ for
$r\ge R_d$ would give $\Vol(P_0+Z_r)=O_d((n+r)^d)$, while each transition
between nonempty residual sets decreases this volume by
$\Omega_d(r^{\beta_d})$.  Hence the resulting bound is
$O_d((n+r)^d/r^{\beta_d})$; taking $r\asymp n$ yields the desired exponent.

Set $K_i=P_i+Z_n$ and $\Phi_i=\Vol(K_i)$, and write $f_0(Q)$ for the number of
vertices of a polytope $Q$.  The B\'ar\'any--Larman estimate gives
$f_0(Z_n)=\Omega(n^{\beta_d})$ \cite[Theorem~1]{BaranyLarman}, whereas
$\Phi_0=O(n^d)$.  Every $K_i$ is full-dimensional, even if $P_i$ is not, and
$f_0(K_i)\ge f_0(Z_n)$.

If a direction uniquely exposes a vertex of $K_i$, the exposed-face formula
for Minkowski sums implies that it uniquely exposes a vertex of $P_i$.  This
grid point is deleted before $P_{i+1}$ is formed.  Consequently, $K_{i+1}$
contains no vertex of $K_i$, although $K_{i+1}\subset K_i$.  Starting with
$K_{i+1}$ and adjoining the vertices of $K_i$ one at a time gives $f_0(K_i)$
strict inclusions of lattice polytopes.  Because $d!\Vol(K)$ is an integer for
every full-dimensional lattice $d$-polytope $K$, each inclusion increases
volume by at least $1/d!$.  Hence
\[
  \Phi_i-\Phi_{i+1}
  \ge \frac{f_0(K_i)}{d!}
  \ge \frac{f_0(Z_n)}{d!}.
\]
Summing the volume decreases over $i$ gives
$O(n^{d-\beta_d})=O(n^{2d/(d+1)})$ steps for every such $X$.  For the lower
bound on the full grid, Andrews' estimate shows that at any step with a
nonsingleton remaining point set, at most $O(n^{\beta_d})$ points are removed.
The symmetries of the grid ensure that its nonsingleton residual hulls are
full-dimensional.  Since all $n^d$ grid points are eventually removed, the
matching lower bound follows.

\section{Preliminaries and notation}\label{sec:preliminaries}

We follow the notation of \cite{AmbrusHsuPengYan}.  For $Y\subset\R^d$, let
$\conv(Y)$ denote its convex hull.  For a polytope $P$, let $\ext(P)$ denote
its vertex set; for a finite set $Y$, abbreviate
$\ext(Y)=\ext(\conv(Y))$.  Starting from $X=X_0$, define
recursively
\[
  X_i=X_{i-1}\setminus\ext(X_{i-1})
  \qquad (i\ge 1).
\]
The layer number $\tau(X)$ is the smallest $i$ for which $X_i=\varnothing$.
Whenever the initial set $X_0=X$ is understood, put
\[
  P_i=\conv(X_i)
  \qquad (0\le i<\tau(X)).
\]
For a polytope $P$, write $f_0(P)=|\ext(P)|$ for its number of vertices.  We
use $\Vol(P)$ for ordinary Euclidean $d$-volume.
All asymptotic notation is for fixed $d$; its implied constants may depend
only on $d$.

We call a polytope with vertices in $\Z^d$ a \emph{lattice polytope}.  It is
\emph{full-dimensional} if its affine hull is $\R^d$.

For nonempty compact convex sets $A,B\subset\R^d$, their Minkowski sum is
\[
  A+B=\{a+b:a\in A,\ b\in B\}.
\]
If $A=\conv(A_0)$ and $B=\conv(B_0)$ for finite sets
$A_0,B_0\subset\Z^d$, then
\[
  A+B=\conv(A_0+B_0).
\]
Indeed, if $a=\sum_s\alpha_sa_s$ and $b=\sum_t\beta_tb_t$, then
$a+b=\sum_{s,t}\alpha_s\beta_t(a_s+b_t)$, which gives
$A+B\subset\conv(A_0+B_0)$.  The reverse inclusion follows because $A+B$ is
convex and contains $A_0+B_0$.  Thus a Minkowski sum of lattice
polytopes is again a lattice polytope.  For a full-dimensional lattice
$d$-polytope $P$, its normalized
lattice volume is
\[
  \Vol_{\Z}(P)=d!\Vol(P).
\]
Let $B^d$ denote the closed Euclidean unit ball and define
\[
  Z_r=\conv(\Z^d\cap rB^d).
\]
The following is the $k=0$ lower bound in Theorem~1 of B\'ar\'any and
Larman \cite{BaranyLarman}.

\begin{proposition}[B\'ar\'any--Larman]\label{prop:ball-hull}
For every fixed $d\ge 2$ there are constants $b_d>0$ and $R_d\ge 1$ such
that, for every real $r\ge R_d$,
\[
  f_0(Z_r)\ge b_d r^{\beta_d},
  \qquad
  \beta_d=\frac{d(d-1)}{d+1}.
\]
\end{proposition}

For $r\ge 1$, the points $0,\pm e_1,\ldots,\pm e_d$ belong to $Z_r$, where
$e_1,\ldots,e_d$ are the standard basis vectors.  Thus $Z_r$ is
full-dimensional.  It also satisfies
\[
  Z_r\subset[-r,r]^d.
\]

We also use Andrews' lattice-polytope estimate
\cite{Andrews}; see also Theorem~2 of \cite{BaranyLarman}.

\begin{proposition}[Andrews]\label{prop:andrews}
For every fixed $d\ge 2$ there is a constant $a_d>0$ such that every
full-dimensional lattice $d$-polytope $P$ satisfies
\[
  f_0(P)\le a_d\Vol(P)^{(d-1)/(d+1)}.
\]
\end{proposition}

\section{Minkowski sums and lattice volume}\label{sec:lemmas}

For $u,x\in\R^d$, let $\langle u,x\rangle$ denote the standard Euclidean
inner product.  For a nonempty compact convex set $A\subset\R^d$, its support
function and its face exposed by $u\in\R^d$ are
\[
  h_A(u)=\max_{a\in A}\langle u,a\rangle,
  \qquad
  F_A(u)=\{a\in A:\langle u,a\rangle=h_A(u)\}.
\]
We say that $u$ \emph{uniquely exposes} $a\in A$ if $F_A(u)=\{a\}$.

\begin{lemma}[Faces of a Minkowski sum]\label{lem:faces}
For nonempty compact convex sets $A,B\subset\R^d$,
\[
  h_{A+B}(u)=h_A(u)+h_B(u),
  \qquad
  F_{A+B}(u)=F_A(u)+F_B(u).
\]
Consequently, $F_{A+B}(u)$ is a singleton if and only if both $F_A(u)$ and
$F_B(u)$ are singletons.
\end{lemma}

\begin{proof}
Since
\(
  \langle u,a+b\rangle=\langle u,a\rangle+\langle u,b\rangle,
\)
a pair $(a,b)\in A\times B$ is maximizing exactly when $a\in F_A(u)$ and
$b\in F_B(u)$.  This proves both identities.  The sum of two nonempty sets is
a singleton exactly when both sets are singletons.
\end{proof}

\begin{lemma}[Vertex count under Minkowski addition]\label{lem:vertex-count}
Let $Z\subset\R^d$ be a full-dimensional polytope, and let
$P\subset\R^d$ be any nonempty polytope, possibly lower-dimensional.  Then
\[
  f_0(P+Z)\ge f_0(Z).
\]
\end{lemma}

\begin{proof}
Fix $z\in\ext(Z)$.  The directions uniquely exposing $z$ form a nonempty
open cone in $\R^d$, namely
\[
  \langle u,z-z'\rangle>0
  \qquad (z'\in\ext(Z),\ z'\ne z).
\]
Choose $u_z$ in this cone and outside the finitely many proper hyperplanes
\[
  \{u:\langle u,p-p'\rangle=0\},
  \qquad p\ne p'\in\ext(P).
\]
Such a choice exists because a finite union of proper hyperplanes has empty
interior.
Then $u_z$ uniquely exposes some $p_z\in\ext(P)$ and, by
Lemma~\ref{lem:faces}, uniquely exposes $p_z+z$ in $P+Z$.

It remains to check that these vertices are distinct.  Let $y\in\ext(P+Z)$,
and choose $w$ that uniquely exposes $y$.  By Lemma~\ref{lem:faces}, write
$F_P(w)=\{p\}$ and $F_Z(w)=\{z\}$.  If $y=p'+z'$ is any decomposition, then
equality in
\[
  \langle w,y\rangle
  \le h_P(w)+h_Z(w)
\]
forces $p'=p$ and $z'=z$.  Thus every vertex of $P+Z$ has a unique
decomposition as a sum of a point of $P$ and a point of $Z$, and the map
$z\mapsto p_z+z$ is injective.
\end{proof}

\begin{lemma}[Disjointness after one peeling step]\label{lem:expulsion}
Let $X\subset\R^d$ be finite and nonempty, and set
\[
  P=\conv(X),
  \qquad
  X'=X\setminus\ext(X),
  \qquad
  Q=\conv(X').
\]
Suppose $X'\ne\varnothing$.  If $Z$ is a full-dimensional polytope, then
\[
  \ext(P+Z)\cap(Q+Z)=\varnothing.
\]
\end{lemma}

\begin{proof}
Let $y\in\ext(P+Z)$, and choose $u$ uniquely exposing $y$.  By
Lemma~\ref{lem:faces},
\[
  F_P(u)=\{p\},
  \qquad
  F_Z(u)=\{z\},
  \qquad
  y=p+z.
\]
Every vertex of the convex hull of a finite set belongs to that set.  Hence
$p\in\ext(P)\subset X$ and $p\notin X'$.  Moreover, extremality gives
\[
  p\notin\conv(X\setminus\{p\}),
\]
so $p\notin Q$.  Since $Q\subset P$ is compact and $p$ is the unique
maximizer of $u$ on $P$,
\[
  h_Q(u)<h_P(u).
\]
It follows that
\[
  h_{Q+Z}(u)=h_Q(u)+h_Z(u)
  <h_P(u)+h_Z(u)=\langle u,y\rangle.
\]
Thus $y\notin Q+Z$.
\end{proof}

\begin{lemma}[Lattice-volume gap]\label{lem:volume-gap}
Let $L\subset K$ be full-dimensional lattice $d$-polytopes.  If $L$ contains
no vertex of $K$, then
\[
  \Vol(K)-\Vol(L)\ge \frac{f_0(K)}{d!}.
\]
\end{lemma}

\begin{proof}
Order the vertices of a polytope $R$.  Recursively triangulate every facet not
containing the first vertex, using the induced order, and cone the resulting
facet simplices to the first vertex.  This pulling triangulation uses only the
vertices of $R$.  Induction on the dimension shows that the induced
triangulations agree on common faces and that these cones form a triangulation
of $R$.  Hence, when $R$ is a full-dimensional lattice $d$-polytope, it is
decomposed into lattice $d$-simplices.  If $v_0,\ldots,v_d$ are the vertices
of one such simplex, its normalized volume is the positive integer
\[
  \left|\det(v_1-v_0,\ldots,v_d-v_0)\right|\in\Z_{>0}.
\]
Summing over the simplices shows that $\Vol_{\Z}(R)$ is a positive integer.

If $R\subsetneq S$ are full-dimensional compact convex sets, then
$\Vol(S)>\Vol(R)$.  Indeed, choose $x\in S\setminus R$ and a hyperplane that
strictly separates $x$ from $R$, and let
$B\subset\operatorname{int}R$ be a $d$-dimensional ball.
Then $\conv(B\cup\{x\})\subset S$, and its part beyond the separating
hyperplane has positive volume and lies in $S\setminus R$.  Consequently,
strict inclusion between full-dimensional lattice polytopes increases
normalized volume by at least $1$, and ordinary volume by at least $1/d!$.

Order $\ext(K)=\{v_1,\ldots,v_m\}$ and put
\[
  K_j=\conv\bigl(L\cup\{v_1,\ldots,v_j\}\bigr),
  \qquad 0\le j\le m,
\]
where $K_0=L$.  Every $K_j$ is a full-dimensional lattice polytope.  Moreover,
$v_j\notin L$ by hypothesis.  If $v_j\in K_{j-1}$, then the vertex $v_j$ of
$K$ would be a convex combination of points of $K\setminus\{v_j\}$, a
contradiction.  Thus all $m=f_0(K)$
inclusions are strict.  Since $K_m=K$, summing their volume increments proves
the assertion.
\end{proof}

\section{The upper bound}\label{sec:upper}

\begin{theorem}\label{thm:upper}
For every fixed $d\ge 2$, there are constants $C_d>0$ and $N_d\ge1$ such
that every integer $n\ge N_d$ and every nonempty set $X\subset[n]^d$ satisfy
\[
  \tau(X)\le C_d n^{2d/(d+1)}.
\]
\end{theorem}

\begin{proof}
Choose an integer $n\ge\lceil R_d\rceil$, where $R_d$ is given by
Proposition~\ref{prop:ball-hull}, and let $X\subset[n]^d$ be nonempty.  Set
\[
  \tau=\tau(X).
\]
Let $X_0,\ldots,X_\tau$ be the peeling sequence and let
$P_i=\conv(X_i)$ for $0\le i<\tau$.  Fix the same auxiliary polytope
\[
  Z=Z_n=\conv(\Z^d\cap nB^d)
\]
throughout the process, and define
\[
  K_i=P_i+Z,
  \qquad
  \Phi_i=\Vol(K_i).
\]

For $0\le i\le \tau-2$, both $P_i$ and $P_{i+1}$ are nonempty.  Each
$K_j$, $j\in\{i,i+1\}$, is a lattice polytope and contains a translate of
the full-dimensional polytope $Z$.  Thus $K_i$ and $K_{i+1}$ are
full-dimensional, and
\[
  K_{i+1}\subset K_i.
\]
Lemma~\ref{lem:expulsion} says that $K_{i+1}$ contains no vertex of $K_i$.
Lemmas~\ref{lem:volume-gap} and \ref{lem:vertex-count}, together with
Proposition~\ref{prop:ball-hull}, give
\begin{equation}\label{eq:drop}
  \Phi_i-\Phi_{i+1}
  \ge \frac{f_0(K_i)}{d!}
  \ge \frac{f_0(Z)}{d!}
  \ge \frac{b_d}{d!}n^{\beta_d}.
\end{equation}

Summing \eqref{eq:drop} over the $\tau-1$ transitions between nonempty
residual sets gives
\begin{equation}\label{eq:telescope}
  (\tau-1)\frac{b_d}{d!}n^{\beta_d}
  \le \Phi_0-\Phi_{\tau-1}
  \le \Phi_0.
\end{equation}

Since $P_0\subset[1,n]^d$ and $Z\subset[-n,n]^d$,
\[
  P_0+Z\subset[1-n,2n]^d.
\]
Consequently,
\begin{equation}\label{eq:initial-potential}
  \Phi_0\le(3n-1)^d\le(3n)^d.
\end{equation}
Combining \eqref{eq:telescope} and \eqref{eq:initial-potential}, and using
\(
  d-\beta_d=2d/(d+1),
\)
we obtain
\[
  \tau\le 1+\frac{d!3^d}{b_d}n^{2d/(d+1)}
  \le \left(1+\frac{d!3^d}{b_d}\right)n^{2d/(d+1)}.
\]
Thus the assertion holds with
$C_d=1+d!3^d/b_d$ and $N_d=\lceil R_d\rceil$.
\end{proof}

\section{The lower bound}\label{sec:lower}

For completeness, we recall the lower-bound argument from
\cite{AmbrusHsuPengYan}.  We first check that every nonsingleton $P_i$ is
full-dimensional, as required by Proposition~\ref{prop:andrews}.

\begin{lemma}\label{lem:full-dimensional-layers}
Let $d\ge 2$.  Every nonempty residual set $X_i$ in the peeling process of
$[n]^d$ that contains at least two points has full-dimensional affine hull.
Hence the only possible nonempty residual set whose affine hull is not all of
$\R^d$ is a singleton.
\end{lemma}

\begin{proof}
Set
\[
  c=\bigl((n+1)/2,\ldots,(n+1)/2\bigr).
\]
Every $X_i$ is invariant under coordinate permutations and reflections in the
coordinate hyperplanes $x_j=(n+1)/2$, $1\le j\le d$, since affine symmetries
commute with taking convex hulls and preserve extreme points.  Central
inversion about $c$ shows that $c\in\operatorname{aff}(X_i)$.  Write
\[
  \operatorname{aff}(X_i)=c+W,
\]
where $W$ is a linear subspace invariant under signed coordinate permutations,
that is, coordinate permutations and independent coordinate sign changes.
If $X_i$ has at least two points, then $W\ne\{0\}$.  Choose
$w\in W$ with $w_j\ne0$, and let $w'$ be obtained from $w$ by changing the
sign of its $j$th coordinate.  Then $w-w'=2w_je_j\in W$.  Coordinate
permutations imply that every standard basis vector belongs to $W$.  Thus
$W=\R^d$.

If $X_i$ is a singleton, invariance under central inversion forces it to be
$\{c\}$.  In particular, this case can occur only when $n$ is odd.
\end{proof}

\begin{theorem}[Ambrus--Hsu--Peng--Yan]\label{thm:lower}
For every fixed $d\ge 2$,
\[
  \tau([n]^d)=\Omega\!\left(n^{2d/(d+1)}\right).
\]
\end{theorem}

\begin{proof}
Write $\tau=\tau([n]^d)$.  By Lemma~\ref{lem:full-dimensional-layers}, every
nonsingleton $P_i$ is a full-dimensional lattice polytope.  Since
$P_i\subset[1,n]^d$, Proposition~\ref{prop:andrews} gives
\[
  f_0(P_i)
  \le a_d\Vol(P_i)^{(d-1)/(d+1)}
  \le a_dn^{\beta_d}.
\]
After replacing $a_d$ by $\max\{a_d,1\}$, the same estimate also covers a
possible singleton residual set.  Every grid point is deleted exactly once,
and therefore
\[
  n^d
  =\sum_{i=0}^{\tau-1}f_0(P_i)
  \le \tau\max\{a_d,1\}n^{\beta_d}.
\]
It follows that
\[
  \tau\ge \frac{1}{\max\{a_d,1\}}n^{d-\beta_d}
  =\frac{1}{\max\{a_d,1\}}n^{2d/(d+1)}.
\]
\end{proof}

\begin{proof}[Proof of Theorem~\ref{thm:main}]
For $d\ge 2$, combine Theorems~\ref{thm:upper} and \ref{thm:lower}.  We may
take
\[
  c_d=\frac{1}{\max\{a_d,1\}},
  \qquad
  C_d=1+\frac{d!3^d}{b_d},
  \qquad
  N_d=\lceil R_d\rceil.
\]
These quantities depend only on $d$.

For $d=1$, each step removes two endpoints, except possibly for a final
singleton.  Thus $\tau([n]^1)=\lceil n/2\rceil$, and
$c_1=1/2$, $C_1=N_1=1$ suffice.
\end{proof}

\enlargethispage{2\baselineskip}


\begin{thebibliography}{99}

\bibitem{AmbrusHsuPengYan}
G.~Ambrus, A.~Hsu, B.~Peng, and S.~Yan.
The layer number of grids.
\arxiv{2009.13130}, 2020.

\bibitem{AmbrusNielsenWilson}
G.~Ambrus, P.~Nielsen, and C.~Wilson.
New estimates for convex layer numbers.
\emph{Discrete Math.}, 344(7):112424, 2021.
\doi{10.1016/j.disc.2021.112424}.

\bibitem{Andrews}
G.~E. Andrews.
A lower bound for the volume of strictly convex bodies with many boundary
lattice points.
\emph{Trans. Amer. Math. Soc.}, 106(2):270--279, 1963.
\doi{10.1090/S0002-9947-1963-0143105-7}.

\bibitem{BaranyLarman}
I.~B\'ar\'any and D.~G. Larman.
The convex hull of the integer points in a large ball.
\emph{Math. Ann.}, 312(1):167--181, 1998.
\doi{10.1007/s002080050217}.

\bibitem{Dalal}
K.~Dalal.
Counting the onion.
\emph{Random Structures \& Algorithms}, 24(2):155--165, 2004.
\doi{10.1002/rsa.10114}.

\bibitem{DillonVaradarajan}
T.~Dillon and N.~Varadarajan.
Explicit bounds for the layer number of the grid.
\arxiv{2302.04244}, 2023.

\bibitem{HarPeledLidicky}
S.~Har-Peled and B.~Lidick\'y.
Peeling the grid.
\emph{SIAM J. Discrete Math.}, 27(2):650--655, 2013.
\doi{10.1137/120892660}.

\bibitem{RoteRuberSaghafian}
G.~Rote, M.~R\"uber, and M.~Saghafian.
Grid peeling of parabolas.
\emph{Leibniz Int. Proc. Inform.}, 293:76:1--76:18, 2024.
\doi{10.4230/LIPIcs.SoCG.2024.76}.

\end{thebibliography}
\end{document}